\documentclass[12pt]{article}
\usepackage{bbm}
\usepackage{amsmath}
\usepackage[hidelinks]{hyperref}
\allowdisplaybreaks

\renewcommand{\emph}[1]{\textit{#1}}
\usepackage{enumerate,amsmath,amsthm,latexsym,amssymb}
\usepackage{color}\usepackage{graphicx}

\definecolor{brown}{cmyk}{0, 0.72, 1, 0.45}
\definecolor{grey}{gray}{0.5}

\newcommand{\old}[1]{}

\newcounter{rot}%\addtocounter{rot}{1}, \therot

\newcommand{\ignore}[1]{}

\def\cA{{\mathcal A}}
\def\cB{{\mathcal B}}

\newcommand{\set}[1]{\left\{#1\right\}}

\def\cP{\mathcal{P}}

\newcommand{\proofend}{\hspace*{\fill}\mbox{$\Box$}\\ \medskip\\ \medskip}

\def\ii_(#1,#2){i_{#1}^{#2}}

\def\qs{{\bf qs}}

\def\bx{{\bf x}}

\def\b{\beta}
\def\d{\delta}

\def\e{\varepsilon}

\def\G{\Gamma}
\def\k{\kappa}

\def\z{\zeta}

\def\l{\lambda}
\def\m{\mu}
\def\n{\nu}
\def\p{\pi}

\def\r{\rho}

\def\s{\sigma}

\def\om{\omega}

\def\1{{\bf 1}}
\def\0{{\bf 0}}

\newcommand{\whp}{{\bf whp}\xspace}

\def\cE{\mathcal{E}}

\newcommand{\brac}[1]{\left( #1 \right)}

\def\E{{\bf E}}

\renewcommand{\Pr}{\operatorname{\bf Pr}}
\newcommand\bfrac[2]{\left(\frac{#1}{#2}\right)}

\def\bx{{\bf x}}

\def\N{N^*}
\def\2G{{\sc 2greedy}}

\newcommand{\nospace}[1]{}

\def\path{\operatorname{PATH}}

\newtheorem{theorem}{Theorem}[section]

\newtheorem{lemma}[theorem]{Lemma}

\newtheorem{remthm}[theorem]{Remark}

\newcounter{thmtemp}

\usepackage[ruled, linesnumbered]{algorithm2e}
\def\gnm3{G_{n,m}^{\delta\geq 3}}
\def\Gnm3{{\mathcal G}_{n,m}^{\delta\geq 3}}
\newcommand{\beq}[2]{\begin{equation}\label{#1}#2\end{equation}}
\def\G{\Gamma}
\def\cG{{\mathcal G}}
\def\cD{{\mathcal D}}
\def\la{\lambda}
\def\cH{{\mathcal H}}

\newcommand{\jt}[1]{{\color{black}#1}}

\begin{document}
\title{Hamilton cycles in random graphs with minimum degree at least 3: an improved analysis}
\author{ Michael Anastos and Alan Frieze\thanks{Research supported in part by NSF Grant DMS1363136}\\ Department of Mathematical Sciences\\ Carnegie Mellon University\\ Pittsburgh PA15213\\ U.S.A.}

\maketitle
\begin{abstract}
In this paper we consider the existence of Hamilton cycles in the random graph $G=\gnm3$. This a random graph chosen uniformly from $\Gnm3$, the set of graphs with vertex set $[n]$, $m$ edges and minimum degree at least 3. Our ultimate goal is to prove that if $m=cn$ and $c>3/2$ is constant then $G$ is Hamiltonian w.h.p. 
In an earlier paper \cite{Hamd3}, the second author showed that $c\geq 10$ is sufficient for this and in this paper we reduce the lower bound to $c>2.662...$. This new lower bound is the same lower bound found in Frieze and Pittel \cite{FP} for the expansion of so-called Pos\'a sets.
\end{abstract}
\section{Introduction}
In this paper we consider the existence of Hamilton cycles in the random graph $G=\gnm3$. This a random graph chosen uniformly from $\Gnm3$, the set of graphs with vertex set $[n]$, $m$ edges and minimum degree at least 3. 
\jt{If $c=3/2$ then $\gnm3$ is precisely the random 3-regular graph which is proven, via the small cycle conditioning method, to be Hamiltonian \cite{RW}. However as $G_{n,3/2n}^{\delta \geq 3}\not\subset G_{n,cn}^{\delta \geq 3}$ for every $c>0$ we cannot directly infer Hamiltonicity for larger values of $c$. In addition, due to the increase in the variance of the degree sequence, the method itself cannot be transfered directly.}
Our ultimate goal is to prove that if $m=cn$ and $c>3/2$ is constant then $G$ is Hamiltonian w.h.p. In an earlier paper \cite{Hamd3}, the second author showed that $c\geq 10$ is sufficient for this and in this paper we reduce the lower bound to $c>2.662...$. This new lower bound is the same lower bound found in Frieze and Pittel \cite{FP} for expansion of so-called Pos\'a sets i.e. \jt{  sets of endpoints that may be formed via the application of P\'osa rotations, P\'osa \cite{Posa}.} In summary we prove,
\begin{theorem}\label{th1}
W.h.p. $\gnm3$ is Hamiltonian for $m=cn,c>2.662...$.
\end{theorem}

One of the motivations for studying this problem arises from the fact that the 3-core of the random graph $G_{n,m}$ is distributed precisely as $G_{\nu,\mu}^{\delta\geq 3}$, where $\nu,\mu$ are the (random) number of vertices and edges in the 3-core and \jt{ w.h.p. $\nu$ is known to be linear in $n$}. In particular, it is plausible that the first non-empty 3-core in the random graph process is Hamiltonian w.h.p. To prove this to be true, we would need to reduce the lower bound on $c$ to \jt{the edges to vertices ratio of the corresponding 3-core which is known to be w.h.p.\@} about 1.8 \cite{JL}. In addition, we note that Krivelevich, Lubetzky and Sudakov \cite{KLS} showed that w.h.p. the first non-empty $k$-core, $k\geq 15$, is Hamiltonian.

\section{Proof of Theorem \ref{th1}}
\subsection{The game plan}
The key to the proof Theorem \ref{th1} is the following lemma:
\begin{lemma}\label{lem:Ham3}
Let $V=[n]$ and $G=(V,E)$ where $E=E_1\cup E_2$ and $E_2=\{e_1,...,e_a\}\subset \binom{V}{2} \setminus E_1$. Let $G_1=(V,E_1)$ and let  $\cP$ be a set of vertex disjoint paths in $G_1$ that covers $V$. Suppose that for some $0<\b<1$, 
\begin{enumerate}[{\bf P1}]
\item $|\cP|\leq \min\set{\frac{|E_2|}{n^{2-2\b}\log^2n},\frac{n^{\b}}{4\log n}}$.
\item Given $e_1,e_2,\ldots,e_{i-1}$, the edge $e_i$ is chosen uniformly from \jt{ $\binom{V}{2} \setminus \big(E_1 \cup \set{e_1,\ldots,e_{i-1}}\big).$
%$\binom{A_i}{2}$ where $A_i$ is a set of size at least $n-2i$
}
\item $X\subseteq V$, $|X| \leq n^{\beta}$ implies that either \jt{$e(X \cup N(X))\leq |X\cup N(X)|$ or $|N(X)|\geq 2 |X|$.\\
(Here $N(X)=\set{y\in V\setminus X:\exists x\in X\text{ such that }\set{x,y}\in E_1}$. In addition $e(X \cup N(X))$ denotes the number of edges spanned by $X \cup N(X)$.)}
\end{enumerate} 
Then $G$ is Hamiltonian with probability $1-o(n^{-3})$.
\end{lemma}
\begin{proof}
Let $\cP=\{P_1,P_2,\ldots,P_\ell\}$ be a minimum cardinality set of vertex disjoint paths in $G_1$ that covers $V$ (and satisfies {\bf P1}). Let the endpoints of $P_i$ be $v_{(i,1)}$ and $v_{(i,2)}$  for $i \in[\ell]$.  Because $\cP$ is of minimum cardinality we have that  $\set{v_{(i,2)}v_{(i+1,1)}}\notin E_1$ for $i\in [\ell]$ (here we identify $v_{(\ell+1,1)}$ with $v_{(1,1)}$). 
In addition, $H_0=v_{(1,1)},P_1,v_{(1,2)}v_{(2,1)},P_2,v_{(2,2)}v_{(3,1)}P_3,\ldots,v_{(\ell,1)}P_\ell $ $ v_{(\ell,2)}v_{(1,1)}, v_{(1,1)}$ is a Hamilton cycle in the graph $\Gamma_0=(V,E_1\cup R)$ where $R=\{\set{v_{(i,2)}v_{(i+1,1)}}$ $:i\in[\ell]\}$.

Starting with $H_0$, we find a Hamilton cycle in $G$ by removing the edges of $R$ from our cycle. We do this with at most $\ell$ rounds of an extension-rotation procedure. Fix $i\geq 0$ and suppose then that after $i$ rounds, we have a Hamilton cycle $H_i$ in the graph $\G_i=(V, E_1\cup R_i\cup F_i)$ where $R_i\subseteq R$ and $|R_i|\leq \ell-i$. Here $F_i=\set{e_1,e_2,\ldots,e_b}$ are the edges of $E_2$ that have been {\em revealed} so far. We explain revealed momentarily.

We start round $i+1$ by deleting an edge $e$ from $R_i$ to create a Hamilton path $Q_1$. We then use P\'osa rotations to  try to find a Hamilton cycle in $\G_i-e$. Given a path $P=(x_1,x_2,\ldots,x_s)$ and an edge $\set{x_s,x_j}$ where $1<j<s-1$, the path $(x_1,\ldots,x_j,x_s,x_{s-1},\ldots,x_{j+1})$ is said to be obtained from $P$ by a rotation with $x_1$ as the fixed end vertex. The edge $\set{x_s,x_j}$ will be called the rotating edge.

First consider all Hamilton paths obtainable from $Q_1$ by a sequence of rotations with $x_1$ fixed. In these rotations, we are only allowed to use edges from $E(\G_i)\setminus\set{e}$ as rotating edges. Next let $END(Q_1,x_1)$ denote the set of end vertices of these paths, other than $x_1$. If there exists $y\in END(Q_1,x_1)$ such that $\set{x_1,y}\in E(\G_i)\setminus\set{e}$ then this round is complete. We have a Hamilton cycle containing one less member of $R$. Thus we can define $R_{i+1}=R_i\setminus\{e\}$ and $F_{i+1}=F_i$.

In the event there is no such $y$, we proceed as follows: Let $END(Q_1,x_1)=\set{z_1,z_2,\ldots,z_q}$ and let $Q_j,j=2,\ldots,q$ denote a path from $x_1$ to $z_j$ found by rotations. Then, for $1\leq j\leq q$, we let $END(Q_j,z_j)$ denote the set of end vertices of paths obtainable from $Q_j$ by a sequence of rotations with $z_j$ fixed. If for some $j$ we find $y\in END(Q_j,z_j)$ such that $\set{z_j,y}\in E(\G_i)\setminus\set{e}$ then, as before, this round is complete. We have a Hamilton cycle containing one less member of $R$. We can then define $R_{i+1}=R_i\setminus\{e\}$ and $F_{i+1}=F_i$. 

Failing this, we start revealing the edges of $e_{b+1},e_{b+2},...,e_a$, in this order, to search for an edge of the form $\set{z_j,y_j}$ where $y_j\in END(Q_j,z_j)$. If $e_c$ is the first such edge, $b\leq c \leq a$, then we let $R_{i+1}=R_i\setminus\{e\}$,
$F_{i+1}=F_i \cup \set{e_{b+1},e_{b+2},\ldots,e_c}$,  
$\G_{i+1}=(V, E_1\cup R_{i+1}\cup F_{i+1})$ and $H_{i+1}$ be a Hamilton cycle in $\G_{i+1}$. P\'osa's lemma states that $|N(END(Q_j,z_j))|<2|END(Q_j,z_j)|$ (see Corollary 6.7 of \cite{FK}) and \jt{Lemma 2.1 of \cite{FP} that  
$e\big(N(END(Q_j,z_j))\cup END(Q_j,z_j)\big)>|N(END(Q_j,z_j))\cup END(Q_j,z_j)|$. Thus, {\bf P3} implies that $|END(Q_j,z_j)|>n^\beta$ for all $1\leq j\leq q$ and similarly that $q>n^\beta$.}

For $1\leq l\leq a=|E_2|$ let $Y_l$ be the indicator for the event that either $e_l$ is not revealed (in any round) in the above procedure or when it is revealed a new Hamilton cycle is identified. From {\bf P2}, we have,
$$\Pr(Y_j=1)\geq \frac{\binom{n^{\beta}-2j}{2}}{\binom{n}{2}} \geq \frac{n^{2\beta-2}}{5},$$
for $j\leq n^\b/4$.

In the event that $G$ is not Hamiltonian all the edges in $E_2$ are revealed and for less than $|\cP|$ of them a new Hamilton cycle is identified. Indeed, if we assume otherwise then $\G_{|\cP|} \subseteq \G$ is Hamiltonian. Hence,
 $Z\leq |\cP|$. But $Y_l ,1\leq l\leq a$ dominates a $Bernoulli(n^{2\beta-2}/5)$ random variable. This domination holds regardless of $Y_1,Y_2,\ldots,Y_{l-1}$. Hence, from {\bf P1}, we have 
$$\Pr( G\text{ is not Hamiltonian }) \leq \Pr(Binomial(n^{2 -2\beta}|\cP|\log^2 n,n^{2\beta-2}/5) \leq |\cP|) = o( n^{-3}).$$
\end{proof}
\subsection{Choice of $E_2$}
Let
$$s=n^{1/2}\log^{-2}n$$
and let
$$\Omega=\set{(H,Y):H\in \cG_{n,cn-s}^{\d\geq 3},
Y\subseteq\binom{[n]}{2},|Y|=s\text{ and }E(H)\cap Y=\emptyset}$$
where $\cG_{n,m}^{\d\geq 3}=\set{G_{n,m}^{\d\geq 3}}$.

We consider two ways of randomly choosing an element of $\Omega$.
\begin{enumerate}[{\bf (a)}]
\item First choose $G$ uniformly from $\cG_{n,cn}^{\d\geq 3}$ and then choose
an $s$-set $X$ uniformly from $E(G)\setminus E_3(G)$,
where $E_3(G)$ is the set of edges of $G$ that are
incident with a vertex of degree 3. This produces a pair
$(G-X,X)$. We let $\Pr_a$ denote the induced probability
measure on $\Omega$.
\item Choose $H$ uniformly from $\cG_{n,cn-s}^{\d\geq 3}$ and
then choose an $s$-set $Y$ uniformly from $\binom{[n]}{2}\setminus E(H)$.
This produces a pair
$(H,Y)$. We let $\Pr_b$ denote the induced probability
measure on $\Omega$.
\end{enumerate}
The following lemma implies that as far as properties that
happen \whp\ in $G$, we can use Method (b), just as well as Method (a) to
generate our pair $(H,Y)$. For a proof see Lemma 10.1 of \cite{Hamd3}.
\begin{lemma}\label{contig}
There exists $\Omega_1\subseteq \Omega$ such that
\begin{enumerate}[{\bf (i)}]
\item $\Pr_a(\Omega_1)=1-o(1)$.
\item $\om=(H,Y)\in \Omega_1$ implies that $\Pr_a(\om)=(1+o(1))\Pr_b(\om)$.
\end{enumerate}
\end{lemma}
It follows that we can take $E_2$ as the set $Y$ in the lemma and then we have $|E_2|=n^{0.5-o(1)}$ and this covers {\bf P2} of Lemma \ref{lem:Ham3}.
\subsection{P3 of Lemma \ref{lem:Ham3}}
The main result of \cite{FP}, (see Theorem 1.1 of that paper), is that if $m=cn$ and $c>2.6616\ldots$ then w.h.p. if ${e(S\cup N(S))>|S\cup N(S)|}$ then $|S|+|N(S)|\geq n^{1-o(1)}$. So, we see that we can take $\b=0.99$ in Lemma  \ref{lem:Ham3}. This covers {\bf P3}.
 
\jt{In \cite{FP} it is also shown that if $G$ has minimum degree 3, $P$ is a path of $G$ and $x$ an endpoint of $P$ then the set $S=END(P,x)$, defined in the proof of Lemma \ref{lem:Ham3}, satisfies the relation ${e(S\cup N(S))>|S\cup N(S)|}$.}

\jt{{\bf P1} of Lemma \ref{lem:Ham3} will follow from the analysis of \2G in Section \ref{analysis}.
}

\section{Random Sequence Model}\label{refined}

We must now take some time to explain the model we use for $\gnm3$. We use a variation on the pseudo-graph model of Bollob\'as and Frieze \cite{BollFr} and Chv\'atal \cite{Ch}. Given a sequence $\bx = (x_1,x_2,\ldots,x_{2M})\in [N]^{2M}$ of $2M$ integers between 1 and $N$ we can define a (multi)-graph
$G_{\bx}=G_\bx(N,M)$ with vertex set $[N]$ and edge set $\{(x_{2i-1},x_{2i}):1\leq i\leq M\}$. The degree $d_\bx(v)$ of $v\in [N]$ is given by 
$$d_\bx(v)=|\set{j\in [2M]:x_j=v}|.$$
If $\bx$ is chosen randomly from $[N]^{2M}$ then $G_{\bx}$ is close in distribution to $G_{N,M}$. Indeed,
conditional on being simple, $G_{\bx}$ is distributed as $G_{N,M}$. To see this, note that if $G_{\bx}$ is simple then it has vertex set $[N]$ and $M$ edges. Also, there are $M!2^M$ distinct equally likely values of $\bx$ which yield the same graph. 

\jt{We will use the above variation on the pseudo-graph model to analyze \2G, an algorithm that finds 2-matchings, applied to $\gnm3$. A 2-matching is a set of edges such that every vertex is incident to at most 2 edges in it.
 \2G is described in Section \ref{alg}. As \2G progresses vertices become matched (incident with edges selected for the 2-matching), edges are deleted and vertices of small degree are identified. 
 As such we will need to impose additional constrains on the vertex degrees and our situation becomes more complicated. At any step of the algorithm we keep track of 3 sets $J_3,J_2$ and $J_0$ that partition the current vertex set, say $[N]$. (A vertex that becomes incident with 2 edges of the 2-matching is not  included in the current vertex set.) $J_3$ is a set of vertices of degree at least 3 and it consists of vertices that have not been matched yet. $J_2$ is a set of vertices of degree at least 2 and it consists of vertices that are incident to exactly 1 edge in the current 2-matching. Finally $J_0$ consists of the remaining vertices and whose sum of degrees will be proven to be $D=o(N)$.}
%Our situation is complicated by there being lower bounds of $2,3$ respectively on the minimum degree in two disjoint sets $J_2,J_3\subseteq [N]$. Initially $J_2=J_3=\emptyset$ but we will have to consider instances where they are non-empty, as our 2-matching algorithm progresses. (These sets are intrinsic to the algorithm \2G\ described in the next section and a 2-matching is a graph of maximum degree at most 2.) The vertices in $J_0=[N]\setminus (J_2\cup J_3)$ are of fixed bounded degree and the sum of their degrees is $D=o(N)$. 

So we let
$$[N]^{2M}_{J_2,J_3;D}=\{\bx\in [N]^{2M}:d_\bx(j)\geq i\text{ for }j\in J_i,\,i=2,3\text{ and }\sum_{j\in J_0}d_\bx(j)=D\}.$$
Let $G=G(N,M,J_2,J_3;D)$ be the multi-graph $G_\bx$ for $\bx$ chosen uniformly from $[N]^{2M}_{J_2,J_3;D}$. 
\jt{What we need now is a procedure that generates $G_\bx$ conditioned on $G_\bx$ being simple or equivalently a way to access the degree sequence of elements in $[N]^{2M}_{J_2,J_3;D}$. Such a procedure is given in \cite{Hamd3} and it is justified by Lemmas \ref{lem3}, \ref{lem4} and \ref{lem4x} that
follow. In Lemma \ref{lem3} it is proven that the degree sequence of 
$[N]^{2M}_{J_2,J_3;D}$ (restricted to the sets $J_2,J_3$) has the same distribution as the joint distribution of $\cP_1,\cP_2,...,\cP_{|J_2|+|J_3|}$
where (i)for $i\in J_\ell$, $\cP_i$ is a $Poisson(\lambda)$ random variable condition on being at least $\ell$  for some carefully chosen value of $\lambda$ and (ii) $\sum_{i=1}^{|J_2|+|J_3|}\cP_i=2M-D$. In Lemma \ref{lem4} it is shown that the marginal of $d_\bx(j)$ and joint of $(d_\bx(j_1),d_\bx(j_2))$ distributions are close to the marginal of $\cP_i$ and joint of $(\cP_i,\cP_j)$ distributions respectively. This fact is used in Lemma \ref{lem4x} 
where we establish concentration of the number of vertices of degree $k$ in $J_\ell, \ell=2,3$. For the proofs of Lemmas \ref{lem3}, \ref{lem4} and \ref{lem4x} see \cite{Hamd3}.}

%It is clear then that conditional on being simple, $G(n,m,\emptyset,[n];0)$ has the same distribution as $G_{n,m}^{\d\geq 3}$. It is important therefore to estimate the probability that this graph is simple. For this and other reasons, we need to have an understanding of the degree sequence $d_\bx$ when $\bx$ is drawn uniformly from $[N]^{2M}_{J_2,J_3;D}$. 
Let 
$$f_k(\l)=e^\l-\sum_{i=0}^{k-1}\frac{\l^i}{i!}$$
for $k\geq 0$.
\begin{lemma}
\label{lem3}
Let $\bx$ be chosen randomly from $[N]^{2M}_{J_2,J_3;D}$. For $i=2,3$ let $Z_j\,(j\in [J_i])$ be independent copies of a {\em truncated Poisson} random variable $\cP_i$, where
$$\Pr(\cP_i=t)=\frac{{\l}^t}{t!f_i({\l})},\hspace{1in}t=i,i+1,\ldots\ .$$
Here ${\l}$ satisfies
\begin{equation}\label{21}
\sum_{i=2}^3\frac{{\l}f_{i-1}({\l})}{f_i({\l})}|J_i|=2M-D.
\end{equation}
For $j\in J_0$, $Z_j=d_j$ is a constant and $\sum_{j\in J_0}d_j=D$. Then $\{d_\bx(j)\}_{j\in [N]}$ is distributed as $\{Z_j\}_{j\in [N]}$ conditional on $Z=\sum_{j\in [n]}Z_j=2M$.
\end{lemma}
%\proofstart
%This is Lemma 3.1 of \cite{Hamd3}.
%\proofend

To use Lemma \ref{lem3} for the approximation of vertex degrees distributions we need to have  sharp estimates of the probability that $Z$ is close to its mean $2M$. In particular we need sharp estimates of
$\Pr(Z=2M)$ and $\Pr(Z-Z_1=2M-k)$, for $k=o(N)$. These estimates are possible precisely because $\E(Z)=2M$. Using the special properties of $Z$, a standard argument in an appendix of \cite{Hamd3} shows that where $N_\ell=|J_\ell|$ and $\N=N_2+N_3$ and the variances are
\begin{equation}\label{30}
\s_\ell^2=\frac{f_\ell({\l})({\l}^2f_{\ell-2}({\l})+{\l}f_{\ell-1}({\l}))-{\l}^2f_{\ell-1}({\l})^2}
{f_\ell({\l})^2}\text{ and }\s^2=\frac{1}{\N}\sum_{\ell=2}^3N_\ell\s_\ell^2,
\end{equation}
that if $\N\s^2\rightarrow \infty$ and $k=O(\sqrt{\N}\s)$ then
\beq{ll1}{
\Pr\left(Z=2M-k\right)=\frac{1}{\s\sqrt{2\p \N}}\left(1+
O\bfrac{k^2+1}{\N\s^{2}}\right).
}

Given \eqref{ll1} and
$$\s_\ell^2=O({\l}),\qquad\ell=2,3,$$
we obtain
\begin{lemma}
\label{lem4}
Let $\bx$ be chosen randomly from
$[N]^{2M}_{J_2,J_3;D}$.
\begin{description}
\item[(a)] Assume that $\log \N=O((\N {\l})^{1/2})$. For every
$j\in J_\ell$ and $\ell\leq k\leq \log \N$,
\beq{f1}{
\Pr(d_\bx(j)=k)=\frac{{\l}^k}{k!f_\ell({\l})}
\left(1+O\left(\frac{k^2+1}{\N {\l}}\right)\right).
}
Furthermore, for all $\ell_1,\ell_2\in\set{2,3}$
and $j_1\in J_{\ell_1},j_2\in J_{\ell_2},\,j_1\neq j_2$, and
$\ell_i\leq k_i\leq \log \N$,
\beq{f2}{
\Pr(d_\bx(j_1)=k_1,d_\bx(j_2)=k_2)=\frac{{\l}^{k_1}}{k_1!f_{\ell_1}({\l})}\frac{{\l}^{k_2}}{k_2!
f_{\ell_2}({\l})}\left(1+O\bfrac{\log^2 \N}{\N {\l}}\right).
}
%\item[(b)] For all $j\in J_\ell$ and $0\neq \ell\leq k\leq \log \N$
%$$ \Pr(d_\bx(j)=k)=O\brac{(\N {\l})^{1/2}\frac{{\l}^k}{k!f_\ell({\l})}}.$$
\item[(b)]
\beq{maxdegree}{
d_\bx(j)\leq\frac{\log N}{(\log\log N)^{1/2}} \quad\qs\footnote{An event
$\cE=\cE(\N)$
occurs quite surely
(\qs, in short) if $\Pr(\cE)=1-O(N^{-a})$ for any constant $a>0$}
}
for all $j\in J_2\cup J_3$.
\end{description}
\end{lemma}
%\proofstart
%This is Lemma 3.2 of \cite{Hamd3}.
%\proofend

Let $\n_\bx^\ell(s)$ denote the number of vertices in $J_\ell,\ell=2,3$ of degree
$s$ in $G_\bx$.
Equation (\ref{ll1}) and a standard tail estimate for the binomial
distribution shows the following:
\begin{lemma}
\label{lem4x}
Suppose that $\log \N=O((\N {\l})^{1/2})$ and $N_\ell\to\infty$ with $N$.
Let $\bx$ be chosen randomly from
$[N]^{2M}_{J_2,J_3;D}$. Then \qs,
\beq{degconc}{
\cD(\bx)=
\left\{\left|\n_\bx^\ell(j)-\frac{N_\ell {\l}^j}{j!f({\l})}\right|
\leq \brac{1+\bfrac{N_\ell {\l}^j}{j!f({\l})}^{1/2}}\log^2 N,\ k\leq j\leq \log N\right\}.
}
\end{lemma}
\proofend

We can now show $G_\bx$, $\bx\in [n]^{2m}_{\emptyset,[n];0}$ is a good model for
$G_{n,m}^{\d\geq 3}$. For this we only need to show now that
\beq{simpx}{
\Pr(G_\bx\text{ is simple})=\Omega(1).
 }
For this we can use a result of McKay \cite{McK}. If we fix the degree sequence
of $\bx$ then $\bx$ itself is just a random permutation of the multi-graph in which each $j\in [n]$
appears $d_\bx(j)$ times. This in fact is another way of looking at the configuration model of
Bollob\'as \cite{B2}. The reference \cite{McK} shows that the probability $G_\bx$ is simple is
asymptotically equal to $e^{-(1+o(1))\r(\r+1)}$ where $\r=m_2/m$ and
$m_2=\sum_{j\in [n]}d_{\bx}(j)(d_{\bx}(j)-1)$. One consequence of the exponential tails in
Lemma \ref{lem4x} is that $m_2=O(m)$. This implies that $\r=O(1)$ and hence that \eqref{simpx}
holds. We can thus use the Random Sequence Model to prove the occurrence of high probability
events in $G_{n,m}^{\d\geq 3}$. 

All that is left now is to show that we can find a covering collection of paths that satisfy {\bf P1} e.g. $|\cP|\leq n^{0.48}$ will  suffice. For this we need to analyse algorithm \2G\ of \cite{Hamd3}, which is described in Section \ref{alg}.

\section{Greedy Algorithm}\label{alg}
We now describe the algorithm \2G\ of \cite{Hamd3}. Our algorithm will be applied to the random graph $G=G_{n,m}^{\d\geq 3}$ and analyzed in the context of $G_\bx$, \jt{with $N=n$ initially}. As the algorithm progresses, it makes changes to $G$ and we let $\G$ denote the current state of $G$. The algorithm grows a 2-matching $M$ and for $v\in [n]$ we let $b(v)$ be the number of edges in $M$ that are incident to $v$. We let
\begin{itemize}
\item $\m$ be the number of edges in $\G$,
\item $V_{0,j}=\set{v\in [n]:d_\G(v)=0,\,b(v)=j}$, $j=0,1$,
\item $Y_k=\set{v\in [n]:d_\G(v)=k\text{ and }b(v)=0}$, $k=1,2$,
\item $Z_1=\set{v\in [n]:d_\G(v)=1\text{ and }b(v)=1}$,
\item $Y=\set{v\in [n]:d_\G(v)\geq 3\text{ and }b(v)=0}$,\qquad This is $J_3$ of Section \ref{refined}.
\item $Z=\set{v\in [n]:d_\G(v)\geq2\text{ and }b(v)=1}$,\qquad This is $J_2$ of Section \ref{refined}.
\item $M$ is the set of edges in the current 2-matching.
\end{itemize}
{\bf Algorithm }\vspace{-.1in}
\begin{description}
\item[Step 1  $Z_1 \cup Y_1\cup Y_2 \neq\emptyset$]\ \\
Choose a random vertex $v$ from $Z_1 \cup Y_1\cup Y_2$. Let $w$ be a random neighbor of $v$.
(We allow the case $v=w$ as we are analyzing the algorithm within the context of $G_\bx$. This case is of course unnecessary when the input is simple i.e. for $G_{n,m}^{\d\geq k}$).
Add $(v,w)$ to $M$ and delete it from $\G$. Update $b(v)=b(v)+1$, $b(w)=b(w)+1$. Delete all vertices in $V(\G)$ satisfying $b(u)\geq 2$ and the edges incident to them. Delete any isolated vertices. 
\item[Step 2: $Y_1\cup Y_2 \cup Z_1=\emptyset$]\ \\
Choose a random vertex $v$ from $Z \cup Y$. Let $w$ be a random neighbor of $v$.
Add $(v,w)$ to $M$ and delete it to from $\G$. Update $b(v)=b(v)+1$, $b(w)=b(w)+1$. Delete all vertices in $V(\G)$ satisfying $b(u)\geq 2$ and the edges incident to them. Delete any isolated vertices.
\end{description}
The algorithm ends when there are at most $n^{2/5}$ vertices left in $\G$. The output of \2G\ is set of edges in $M$.
\section{Analysis of \2G}\label{analysis}
We will use the following additional notation to that given in Section \ref{alg}:
\begin{itemize}
\item $m_i$: number of edges at time $i$.
\item $Z_j,j\geq 2$ and $Y_j,j\geq 3$ resp. are the subsets of $Z$ and $Y$ respectively constisting of vertices of degree $j$.
\item $y_i=|Y|,z_i=|Z|$ at time $i$.   
\item $\zeta_i= |Y_1|+2|Y_2|+|Z_1|$.
\item $$p_{2,i} =\frac{2|Z_2|}{2m_i} \text{ and } p_{3,i}=\frac{3|Y_3|}{2m_i}.$$
\end{itemize}
Let $\epsilon=10^{-5}$. We also define the stopping time 
$$\tau:=\min\{i:m_i\leq n^{0.4+2\e} \}.$$
We will show that w.h.p.
\beq{goal}{
\text{ for $i<\tau$ we have $\zeta_i< n^{0.4+\e}=o(m_i)$}.
}
Every component in $M$ defines a path and the union of the vertices of these paths is $V$. The number $\k$  of components of the 2-matching $M$ output by \2G\ can be bounded as follows. $\k$ can be bounded by the number $\k_1$ of vertices of degree one or zero in $M$ plus $\k_2$, the number of cycles. For every vertex $v\in V$ that contributes to $\k_1$ there exists a step $i$ such that either (i) $v\in Z_1 \cup Y_1\cup Y_2$ and at step $i$ a neighbor of $v$ is  matched and then removed from $\G$ or (ii) $v\notin Z_1 \cup Y_1\cup Y_2$, 2 neighbors of $v$ are matched and then removed from $\G$ and as a result at least $d(v)-2$ edges incident to $v$ are removed. If the above occurs then we say that step $i$ witnesses an increase of $\k_1$. 

For the number of cycles spanned by $M$, observe that  at step $i$, $\k_2$ can increase by one only if we add an edge $\{u,v\}$ to $M$ where $u$ is connected to $v$ by a path in $M$. If the above occurs then we say that step $i$ witnesses an increase of $\k_2$. 

Since w.h.p the maximum degree of $G_0$, and hence of $\G$, is $\log n$ we have that step $i$ witnesses an increase of $\k_1+\k_2$ of magnitude at most $2\log n$ with probability at most   $(2\log n) \zeta_i/2m_i + O(1/m_i)$.  If $\k_1+\k_2$ reaches $n^{0.4+2\e}$ before time $\tau$ then, there are at least $\e^2 n^{0.4+2\e}/2\log n$ steps with $m_i\in [n^{0.4+2\e+(r-1)\epsilon^2}, n^{0.4+2\e+r\epsilon^2}]$ for some integer $1\leq r\leq 1/\e^2$ that witness an increase of $\k_1+\k_2$. The probability that this occurs for a fixed $r$, while $\z_i\leq n^{0.4+\e}$,  is bounded by
\begin{align*}
\binom{n^{0.4+2\e+r\e^2}}{\frac{\e^2 n^{0.4+2\e}}{2\log n}} \bfrac{2n^{0.4+\e}\log n}{n^{0.4+2\e+(r-1)\e^2}}^{\frac{\e^2 n^{0.4+2\e}}{2\log n}} & \leq \brac{\frac{en^{0.4+2\e+r\e^2}}{\frac{\e^2 n^{0.4+2\e}}{2\log n}}  \cdot \frac{2n^{0.4+\e}\log n}{n^{0.4+2\e+(r-1)\e^2}} }^{\frac{\e^2 n^{0.4+2\e}}{2\log n}} \leq n^{-5}.
\end{align*} 
Hence w.h.p. if  $\z_i\leq n^{0.4+\e}$ for $i<\tau$ then the total increase in $\k_1+\k_2$ in the first $\tau-1$ steps is bounded by $n^{0.4+2\e}$. Once $m_i\leq n^{0.4+2\e}$, at most $n^{0.4+2\e}$ more components can be created, yielding in total at most $2n^{0.4+2\e}$ components.

For $i<\tau$, we define the events
\[ 
\cA_i=\set{(z_j+y_j)\la_j \geq \log^3 n \text{ for } j\leq i}\text{  and }\cB_i=\set{(\lambda_i\geq m_i^{-0.2})\vee (y_i\geq m_i^{0.8})}.
\]
For $i<\tau$, we also define the following random variables:
\begin{align*}
X_i&=(\zeta_{i+1}-\zeta_i) \mathbb{I}(\cA_i, \cB_i, 0<\zeta_i<n^{0.4+\e}).\\
Y_i&=(\zeta_{i+1}-\zeta) \mathbb{I}(\cA_i,\neg \cB_i, 0<\zeta_i<n^{0.4+\e}).\\
X_i'&=(\zeta_{i+1}-\zeta) \mathbb{I}(\neg \cA_i, \cB_i, 0<\zeta_i<n^{0.4+\e}).\\
Y_i'&=(\zeta_{i+1}-\zeta) \mathbb{I}(\neg \cA_i, \neg\cB_i, 0<\zeta_i<n^{0.4+\e}).
\end{align*}
For $0<i<\tau$ we have that w.h.p.
\begin{align}\label{goaly}
\min\{ \zeta_{i}, n^{0.4+\e}\} \leq M+\sum_{j=0}^{i-1}(X_i+Y_i+X_i'+Y_i')
\end{align}
where $M=\log^2 n$ is such that the following holds: w.h.p. for every $i\geq 0$ with $\zeta_i=0$ we have that $\zeta_{i+1} \leq M$. Our bound for $M$ is justified by the fact that the maximum degree in $G$ is $o(\log n)$ w.h.p.

%We now prove high probability upper bounds on the random variables in \eqref{goaly} and only consider $i$ such that \beq{mi}{m_i\geq n^{0.4+2\e}.}

We use the inequality $i<\tau$, hence $m_i\geq n^{0.4+2\e}$, to impose that if $\zeta_i\leq n^{0.4+\e}$ then almost all of the vertices belong to $Y \cup Z$. We will see from the analysis below that w.h.p.
\beq{important}{
m_i\geq n^{0.4+2\e}\text{ implies }\zeta_i\leq n^{0.4+\e}.
}

Equation (80) of \cite{Hamd3} states that if $\cH_i$ denotes the history of the process up to the end of iteration $i$, assuming the event $\cA_i$ occurs,   then
\beq{neg}{
\zeta_i>0\text{ implies }\E(\zeta_{i+1}-\zeta_i\mid \cH_i)\leq -\Omega(\min\set{1,\la_i}^2)+ O\bfrac{\log^2m_i}{\la_im_i}.
}
In the following cases we will assume that $i<\tau$ and $\zeta_i>0$. The case $\z_i=0$ is handled by $M$ of \eqref{goaly}.

\textbf{Case 1: $\cA_i \wedge \cB_i$}\\
 \textbf{Case 1a}\\
If $\lambda_i\geq m_i^{-0.2}$ we have from \eqref{neg} that  
\[
\E(X_i|\cH_i) \leq - c\lambda_i^2 \leq -cn^{-0.4}
\]
for some constant $c>0$.

\textbf{Case 1b:}\\
Assume now that $\la_i \leq m_i^{-0.2}$. In this case since $\cA_i$ occurs we have that for $i\geq 2$, $|Z_i|$ is approximately equal to the sum of $|Z_i|$ independent random variables that follow Poisson($\lambda_i$) conditioned on having value at least 2. More precisely, it follows from Lemma 3.3 of \cite{Hamd3} that  as long as $\cA_i$ holds, we have 
\begin{equation}\label{1}
  \begin{split}
  \frac{|Z_3|}{|Z_2|}&= \frac{\lambda_i}{3}\brac{1+O(m_i^{1/2}\lambda_i\log^2m_i)},\\
 \frac{|Z_4|}{|Z_2|}&= \frac{\lambda_i^2}{12}\brac{1+O(m_i^{1/2}\lambda_i\log^2m_i)},\\
\sum_{i\geq 5} |Z_i|& \leq |Z_2| \lambda_i^3.
    \end{split}
\end{equation}
Similarly 
\begin{equation}\label{2}
  \begin{split}
  \frac{|Y_4|}{|Y_3|}&= \frac{\lambda_i}{4}\brac{1+O(m_i^{1/2}\lambda_i\log^2m_i)},\\
 \frac{|Y_5|}{|Y_3|}&= \frac{\lambda_i^2}{20}\brac{1+O(m_i^{1/2}\lambda_i\log^2m_i)},\\
\sum_{i\geq 6} |Y_i|& \leq |Y_3| \lambda_i^3.
    \end{split}
\end{equation}
Recall that if $\zeta_i>0$ then the algorithm will  choose a vertex  $v\in Z_1 \cup Y_1 \cup Y_2$ and it will match it to some vertex $w$. Thus initially $\zeta_{i}$ will decrease by 1.

 For $w\in Z$ let $d(w,Y_3)$ and $d(w,Z_2)$ be the number of neighbors of $w$ in $Y_3$ and $Z_2\setminus \{v\}$. Also let $f(w)$ be the number of vertices that are connected to $w$ by multiple edges. We consider the following cases:

 \textbf{Case a:} $w\in Y_2 \cup Y_1 \cup Z_1$ then  $\zeta_{i+1}-\zeta_i=-2$.
%\\ \textbf{Case b:} $w \in Y_2$ then $\zeta_{i+1}-\zeta_i=-1$.
\\ \textbf{Case b:} $w \in Y$ then $\zeta_{i+1}-\zeta_i=-1$.
%
%If $w\in Z_2$ then $\zeta_{i+1}-\zeta_i=-1+d(w,Z_2)+2d(w,Y_3)+O(f(w))$.\\
\\ \textbf{Case c:} $w\in Z_2$ and $d(w,Z_2)=1$ then $\zeta_{i+1}-\zeta_i=0$.
\\ \textbf{Case d:} $w\in Z_2$ and $d(w,Y_3)=1$ then $\zeta_{i+1}-\zeta_i=1$.
\\ \textbf{Case e:} $w\in Z_2$ and $d(w,Z_2)+d(w,Y_3)=0$ then $\zeta_{i+1}-\zeta_i=-1$.
\\ \textbf{Case f:} $w\in Z \setminus Z_2$ then $\zeta_{i+1}-\zeta_i\leq -1+d(w,Z_2)+2d(w,Y_3)+O(f(w))$.

Differentiating cases c,d,e,f will be helpful later when we bound $\sum_{i\geq 0} Y_i$.

 Summarizing we have, 
\begin{equation}\label{cases}
    \zeta_{i+1}- \zeta_i 
    \begin{cases}
     = -2, & \text{ Case a: probability } (\zeta_i/2m_i)(1+O(m_i^{-1})) .\\
     = -1, & \text{ Case b: probability } p_{3,i} (1+O(m_i^{-1})).\\
     =  0, & \text{ Case c: probability } p_{2,i}^2  (1+O(m_i^{-1})).\\
     =  +1 & \text{ Case d: probability } p_{2,i}p_{3,i}(1+O(m_i^{-1}).\\   
	 =  -1 & \text{ Case e: probability } p_{2,i}(1-p_{2,i}-p_{3,i})(1+O(m_i^{-1})).\\  
     \leq -1+d(w,Z_2)\\ \ +2d(w,Y_3)+O(f(w)) & \text{ Case f: }
    \end{cases}
  \end{equation}
  The net contribution of Cases c,d,e  to $\E(X_i|\cH_i)$ is 
\beq{def}{
-p_{2,i}+p_{2,i}(p_{2,i}+2p_{3,i})= -\Pr(w \in Z_2) +p_{2,i}(p_{2,i}+2p_{3,i}).
}  
Similarly, the contribution of Case f  to $\E(X_i|\cH_i)$ is at most
\begin{align}\label{problematic}
&\E\big[-1+(d(w)-1)(d(w,Z_2)+2d(w,Y_3))+O(f(w)))\mathbb{I}(w\in Z\setminus Z_2) |\cH_i\big]\nonumber
\\ &=-\Pr(w\in Z\setminus Z_2)
+ \brac{(3-1) \frac{3|Z_3|}{2m_i}+(4-1)\frac{4|Z_4|}{2m_i}}(p_{2,i}+2p_{3,i})\nonumber \\&+ O\brac{\frac{\lambda_i\log^2m_i}{m_i^{1/2}} + \lambda_i^3}\nonumber\\
&= -\Pr(w\in Z\setminus Z_2)+p_{2,i}\bigg(\lambda_i +\frac{\lambda_i^2}{2}  \bigg)(p_{2,i}+2p_{3,i}) 
+ O\brac{\frac{\lambda_i\log^2m_i}{m_i^{1/2}} + \lambda_i^3}.
\end{align}
The -1 in the $d(w)-1$ expression accounts for the edge $\set{v,w}$. Then the next term accounts for the other $d(w)-1$ neighbors of $w$ and the possibility that they belong to either $Z_2$ or $Y_3$. To go from the second to the third line we used \eqref{1}.

Finally observe that \eqref{1}, \eqref{2} imply that
\begin{align}\label{total}
1&= \frac{2|Z_2|+3|Z_3|+4|Z_4|}{2m_i}+\frac{3|Y_3|+4|Y_4|+5|Y_5|}{2m_i}+\frac{\zeta_i}{2m_i}+ O\brac{\frac{\lambda_i\log^2m_i}{m_i^{1/2}}+\lambda_i^3} \nonumber
\\&= p_{2,i}\bigg(1+ \frac{\lambda_i}{2}+ \frac{\lambda_i^2}{6} \bigg)
+ p_{3,i}\bigg(1 +\frac{\lambda_i}{3}+ \frac{\lambda_i^2}{12}\bigg)+\frac{\zeta_i}{2m_i} + O\brac{\frac{\lambda_i\log^2m_i}{m_i^{1/2}}+\lambda_i^3}.
\end{align}
Therefore,
\begin{align}
\E(X_i|\cH_i)&\leq  \brac{-\frac{2\zeta_i}{2m_i}- \Pr(w \in Y) + [- \Pr(w \in Z_2) + p_{2,i}(p_{2,i}+2p_{3,i})]} \brac{1+O\bfrac{1}{m_i}}\nonumber\\
&+\brac{ -\Pr(w \in Z \setminus Z_2)+ p_{2,i}\bigg(\lambda_i +\frac{\lambda_i^2}{2}  \bigg)(p_{2,i}+2p_{3,i}) } + O\brac{\frac{\lambda_i\log^2m_i}{m_i^{1/2}}+\lambda_i^3}\nonumber\\
&= -1-\frac{\zeta_i}{2m_i} 
+ p_{2,i}\brac{1+\lambda_i +\frac{\lambda_i^2}{2}  }(p_{2,i}+2p_{3,i}) +\nonumber\\ &O\brac{\frac{\lambda_i\log^2m_i}{m_i^{1/2}}+\lambda_i^3}.\qquad\text{Note that $\cA_i$ implies that $\frac{1}{m_i}\ll \frac{\lambda_i\log^2m_i}{m_i^{1/2}}$}.\nonumber\\
\noalign{Now use \eqref{total} to replace -1 by the squared expression to obtain}
& \leq - \bigg[ p_{2,i}\bigg(1+ \frac{\lambda_i}{2}+ \frac{\lambda_i^2}{6} \bigg)
+ p_{3,i}\bigg(1+ \frac{\lambda_i}{3}+ \frac{\lambda_i^2}{12}\bigg)+\frac{\zeta_i}{2m_i} \bigg]^2 \nonumber\\
&  + p_{2,i}\bigg( 1+ \lambda_i+  \frac{\lambda_i^2}{2}  \bigg)(p_{2,i}+2p_{3,i})-  \frac{\zeta_i}{2m_i} + O\brac{\frac{\lambda_i\log^2m_i}{m_i^{1/2}}+\lambda_i^3}\nonumber\\
&= - \frac{\lambda_i^2 p_{2,i}^2}{12}+2 p_{2,i}p_{3,i} \bigg( \frac{\lambda_i}{6} +\frac{\lambda_i^2}{12} \bigg)
- p_{3,i}^2 \bigg(1+ \frac{2\lambda_i}{3} +\frac{5\lambda_i^2}{18} \bigg) -\frac{\zeta_i}{2m_i}\nonumber\\
&+O\brac{\frac{\lambda_i\log^2m_i}{m_i^{1/2}}+ \lambda_i^3}\nonumber\\
& = - \brac{\frac{ \lambda_i p_{2,i}}{4} - p_{3,i}\brac{ \frac{2}{3} +\frac{\lambda_i}{3} } }^2   - \frac{\lambda_i^2 p_{2,i}^2}{48} -p_{3,i}^2\brac{\frac{5}{9} +\frac{2\la_i}{9}+\frac{\la_i^2}{6}} -\frac{3\zeta_i}{2m_i}\nonumber\\
&+O\brac{\frac{\lambda_i\log^2m_i}{m_i^{1/2}}+\lambda_i^3}\nonumber\\
& \leq - \frac{\lambda_i^2 p_{2,i}^2}{48} -\frac{5p_{3,i}^2}{9}  -\frac{\zeta_i}{2m_i}+ O\brac{\frac{\lambda_i\log^2m_i}{m_i^{1/2}}+\lambda_i^3}.  \label{EX}
\end{align}
In Case 1b we have that  the events $\cA_i \wedge \cB_i$ and $\lambda_i \leq m_i^{-0.2}$ occur. In addition $i<\tau$, hence $m_i\geq n^{0.4+2\e}$ occur. 
$\cA_i\wedge \cB_i$ and $\lambda_i \leq m_i^{-0.2}$ imply that $y_i\geq m_i^{0.8}$ and so  $p_{3,i}+p_{2,i}= \Omega(1)$ and $p_{3,i}\geq m_i^{-0.2}$. 
Therefore 
\[
\E(X_i|\cH_i ) \leq -c' m^{-0.4}_i \leq -c n^{-0.4}.
\]   
Thus if Case 1 occurs we have by the Azuma inequality that
\[
\sum_{\ell\geq 0} \Pr\bigg( \sum_{i=0}^j X_i \geq n^{0.4+\e/2} \bigg)
 \leq m_0 \max_{0\leq j\leq m_0}  \exp \bigg\{- \frac{(n^{0.4+\e/2}+ cjn^{-0.4})^2}{j\log^2n}  \bigg\} +n^{-6} =o(1).
\]
The $n^{-6}$ term accounts for the probability that the degree of $G$ exceeds
$\log n$. The maximum degree bounds $|\zeta_{i+1}-\zeta_i|$.

\textbf{Case 2: $\cA_i\wedge\neg\cB_i$}\\ 
To bound $\sum_{i\geq 0}^j Y_i$, let $R_i$ be the indicator of the event that 
%$\neg \cB_i \wedge 
$\set{\zeta_i\leq n^{0.4+\e}}$ plus
one of the cases (a),(b),(d),(e) and(f) from \eqref{cases} occurs. Then, just as in Case 1, since the contribution of Case c to $\E(X_i |\cH_i)$ is 0 and $Y_i=0$ if $\zeta_i\geq n^{0.4+\e}$, we have that
\begin{align}\label{EY}
\E(Y_i R_i |\cH_i )  &\leq - \frac{\lambda_i^2 p_{2,i}^2}{48} -\frac{5p_{3,i}^2}{9}  -\frac{\zeta_i}{2m_i} + O\brac{\frac{\lambda_i\log^2m_i}{m_i^{1/2}}+\lambda_i^3} \nonumber 
\\&\leq - \frac{\lambda_i^2 p_{2,i}^2}{48} + O\brac{\frac{\lambda_i\log^2m_i}{m_i^{1/2}}+\lambda_i^3} \nonumber
\\& \leq O(m_i^{-1}\log^4 m_i).
\end{align}
For the last inequality we used that in the event $\mathcal{A}_i \wedge \neg \mathcal{B}_i$ \eqref{1}, \eqref{2} and \eqref{total} imply that $p_{2,i}=1-o(1)$. In addition 
\begin{align}
\Pr(R_i=1)&\leq\Pr(\text{Case(a)})+\Pr(\text{Case(b)})+ \Pr(\text{Case(d)})+\Pr(\text{Case(e)})+\Pr(\text{Case(f)})\nonumber\\
& =O\bigg( \frac{\zeta_i}{2m_i}+ p_{3,i}+p_{2,i}p_{3,i}+p_{2,i}(1-p_{3,i}-p_{2,i})+\lambda_i\bigg)=O\bigg( \frac{\zeta_i}{2m_i} +p_{3,i}+\lambda_i\bigg). \label{oneof}
\end{align}
where we have used $1-p_{3,i}-p_{2,i}=O(\la_i)$.

In the event $\neg \cB_i$ we have that $\lambda_i \leq m^{-0.2}$ and $y_i \leq m_i^{0.8}$ and hence $p_{3,i} \leq m_i^{-0.2}$. Hence, if $\zeta_i\leq n^{0.4+\e}$ then $\Pr(R_i=1) \leq m_i^{-0.2}$.
% Let $\bar{R_i}=R_i\mathbb{I}(\zeta_i < n^{0.4+\e})$, then,
Thus,
\begin{align*}
\sum_{j=0}^{m_0} \Pr\brac{\sum_{i=0}^j {R_i}    >n^{0.8+\e/3} }
& \leq \sum_{j=0}^{m_0} \Pr\brac{\sum_{i=0}^j {R_i}\mathbb{I}(m_i>n^{0.8})    >n^{0.8+\e/3}-n^{0.8} } 
\\&\leq m_0 \exp\bigg\{ -\frac{(n^{0.8+\e/3}-n^{0.8}-\sum_{m_i =n^{0.8}}^{m_0} m_i^{-0.2})^2}{2m_0} \bigg\} =o(1).
\end{align*}
To obtain the exponential bound, we let $Z_j=\sum_{i=0}^j {R_i}\mathbb{I}(m_i>n^{0.8})$. We have\\
 $\E{Z_j}\leq \sum_{m_i =n^{0.8}}^{m_0} m_i^{-0.2}=O(n^{0.8})$ and then we can use the Chernoff bounds, since our bounds for ${R_i}=1$ hold given the history of the process so far.

It follows that, 
\begin{align}\label{Y}
\sum_{j=0}^{m_0} \Pr&\bigg( \sum_{i=0}^j Y_i \geq n^{0.4+\e/2} \bigg)= \sum_{j=0}^{m_0} \Pr\bigg( \sum_{i=0 }^j Y_i {R_i} \geq n^{0.4+\e/2}\bigg) \nonumber
\\& \leq  \sum_{j=0}^{m_0} \Pr(\sum_{i=0}^j {R_i}  >n^{0.8+\e/3})
+\sum_{j=0}^{m_0} \Pr\bigg( \sum_{i=0 }^j Y_i {R_i} \geq n^{0.4+\e/2}|\sum_{i=0}^j {R_i}  \leq n^{0.8+\e/3} \bigg) \nonumber
\\&\leq o(1) + m_0 
 \max_{j\leq n^{0.8+\e/3}}  \exp\bigg\{-\frac{\big(n^{0.4+\e/2}-\sum_{m_i =0}^{m_0} m_i^{-1}\log^3 m_i\big)^2}{j\log ^2n}\bigg\} \nonumber
\\& \leq o(1)+ m_0 \max_{j\leq n^{0.8+\e/3}}  \exp\bigg\{-\frac{\big(n^{0.4+\e/2}-n^{o(1)}\big)^2}{j\log^2 n}\bigg\} =o(1).
\end{align}
To obtain the third line we use the fact that w.h.p. $|Y_i|\leq \log n$, which follows from a high probability bound of $o(\log n)$ on the maximum degree of $G$.

\textbf{Cases 3 \& 4: $\neg\cA_i$ } \\
Let $T_1=\max\set{i < \tau : \cA_i\text{ occurs}}$. At time $T_1$ we have $(z_{T_1}+y_{T_1})\la_{T_1}\geq  m_{T_1} \log^3 n$ and hence the estimates \eqref{1}, \eqref{2} hold.
Thereafter $|z_{T_1+1}-z_{T_1}|, |y_{T_1+1}-y_{T_1}|, |m_{T_1+1}-m_{T_1}| = O( \Delta(G_{T_1-1}))$. The maximum degree of $\Delta(G_{T_1})$ is bounded  w.h.p. by $\log n$. At time $T_1+1$ we have $(z_{T_1+1}+y_{T_1+1})\la_{T_1+1}< m_{T_1+1} \log^3 n$ hence $\la_{T_1}\leq \frac{2\log^3n}{m_{T_1}}$ and so  subsequently for $i\geq T_1$ we have
\beq{smallsets}{
|Y_4|,|Z_3|=O(\log^3n)\text{ and }Y_j=Z_{j-1}=\emptyset\text{ for }j\geq 5.
}

\textbf{Case 3: $\neg\cA_i \wedge B_i$} \\
Given the above we replace \eqref{total} by
\beq{total1}{
1=p_{2,i}+p_{3,i}+\frac{\zeta_i}{2m_i}+O\bfrac{\log^3n}{m_i}.
}
Following this we replace \eqref{EX} by
\beq{EX1}{
\E(X_i'\mid\cH)\leq -\frac{5p_{3,i}^2}{9} +O\bfrac{\log^3n}{m_i}.
}
In the events $\neg\cA_i \wedge \cB_i$, $y_i\geq m_i^{0.8}$ and so  $p_{3,i}\geq m_i^{-0.2}$. 
Therefore 
\[
\E(X_i'|\cH_i ) \leq -c' m^{-0.4}_i \leq -c n^{-0.4}.
\]   
Thus if Case 3 occurs we have by the Azuma inequality that
\[
\sum_{\ell\geq 0} \Pr\bigg( \sum_{i=0}^j X_i' \geq n^{0.4+\e/2} \bigg)
 \leq m_0 \max_{0\leq j\leq m_0}  \exp \bigg\{- \frac{(n^{0.4+\e/2}+ cjn^{-0.4})^2}{j\log^2n}  \bigg\} +n^{-6} =o(1).
\]
The $n^{-6}$ term accounts for the probability that the degree of $G$ exceeds
$\log n$. The maximum degree bounds $|\zeta_{i+1}-\zeta_i|$.

\textbf{Case 4: $\neg\cA_i \wedge \neg B_i$} \\
As in Case 2 we have
\[
\E(Y_i' R_i |\cH_i ) \leq O(m_i^{-1} \log^4 n)
\]   
where $R_i$ is defined exactly as in Case 3. 
Hence, just as in \eqref{Y} we get
$$\sum_{j=0}^{m_0} \Pr\bigg( \sum_{i=0}^j Y_i '\geq n^{0.4+\e/2} \bigg)=o(1).$$
The above analysis and equation \eqref{goaly} shows that w.h.p. 
$$\min\{\zeta_i,n^{0.4+\e}\}\leq \log^2 n +4n^{0.4+\e/2} <n^{0.4+0.9\e}.$$
Hence w.h.p. there does not exist $i<\tau$ such that $\zeta_i> n^{0.4+\e}$. And this therefore completes the proof that w.h.p. for $i<\tau$ we have $\zeta_i\leq n^{0.4+\e}$, verifying \eqref{goal}.
%\begin{remark}\label{remL}
%  We conclude by observing that as long as we start \2G\ with a partial matching $M$ and a random graph $G$ containing $M$ where $\zeta=0$, then w.h.p. $\zeta_i\leq n^{0.41}$ up to the point where $m_i\leq n^{0.42}$.
%\end{remark}
%We will find this observation useful in \cite{LP} where we study the length of the longest path in a sparse random graph.

\section{Conclusion}
We have made significant progress in determining the number of random edges needed for Hamiltonicity when we condition on minimum degree at least three. Further progress will lie on improving the bound on the number of edges needed to apply P\'osa's theorem that is given in \cite{FP}. This may not be so easy, as explained in Remark 4.1 of \cite{FP}.

\end{document}